\documentclass{amsart}
\usepackage{amsmath, amsthm, amssymb}
\usepackage[colorlinks=true,linkcolor=blue,urlcolor=blue,unicode=true]{hyperref}
\usepackage[capitalize]{cleveref}
\usepackage{tikz-cd}

\theoremstyle{plain}
\newtheorem{thm}{Theorem}[section]

\newtheorem{lem}[thm]{Lemma}
\newtheorem{prop}[thm]{Proposition}

\theoremstyle{definition}
\newtheorem{defi}[thm]{Definition}
\newtheorem{ex}[thm]{Example}
\newtheorem{rem}[thm]{Remark}

\DeclareMathOperator{\coker}{coker}
\DeclareMathOperator{\im}{im}

\DeclareMathOperator{\id}{id}

\DeclareMathOperator*{\colim}{colim}
\DeclareMathOperator{\Hom}{Hom}

\DeclareMathOperator{\rad}{rad}

\newcommand\CH{\check{H}}

\title{Structure of semi-continuous q-tame persistence modules}
\author{Maximilian Schmahl}
\address{Maximilian Schmahl, Mathematisches Institut, Universit\"at Heidelberg}
\email{\href{mailto:mschmahl@mathi.uni-heidelberg.de}{mschmahl@mathi.uni-heidelberg.de}}
\date{}

\AtBeginDocument{%
   \def\MR#1{}
}
\begin{document}

\begin{abstract}
Using a result by Chazal, Crawley-Boevey and de Silva concerning radicals of persistence modules, we show that every lower semi-continuous q-tame persistence module can be decomposed as a direct sum of interval modules and that every upper semi-continuous q-tame persistence module can be decomposed as a product of interval modules.
\end{abstract}

\maketitle

\section{Introduction}
Motivated by the development of topological data analysis, in particular persistent homology, as well as by applications in symplectic topology, there has in recent years been some theoretical interest in certain algebraic structures called persistence modules.

For us, a \emph{persistence module} is a functor $M\colon\mathbf{T}\to\mathbf{Vec}_{\mathbb{F}}$, where $\mathbb{F}$ is some field and $\mathbf{T}$ is the category corresponding to some totally ordered set $(T,\leq)$. $M$ is called \emph{pointwise finite dimensional (p.f.d.)} if $M_{t}$ is a finite dimensional vector space over $\mathbb{F}$ for all $t\in T$ and it is called \emph{q-tame} if the linear map $M_{s,t}$ has finite rank for all $s,t\in T$ with $s<t$. We call $M$ \emph{ephemeral} if $M_{s,t}=0$ for all $s,t\in T$ with $s<t$.

The category of persistence modules is the functor category $\mathbf{Vec}_{\mathbb{F}}^{\mathbf{T}}$. This category is abelian since $\mathbf{T}$ is small and the category of vector spaces is abelian. Kernels, cokernels, direct sums, products, etc. are all given by their pointwise analogues.

One of the most important questions in the theory of persistence is when a given persistence module can be decomposed into elementary building blocks, namely interval modules: For $I\subseteq T$ an interval, define a persistence module $C(I)$ via 
\[
C(I)_t=
\begin{cases}
    \mathbb{F} & \text{if } t\in I,\\
    0              & \text{otherwise,}
\end{cases}
\]
with structure maps 
\[
C(I)_{s,t}=
\begin{cases}
    \id_{\mathbb{F}} & \text{if } s,t\in I,\\
    0              & \text{otherwise.}
\end{cases}
\]
Such persistence modules are called \emph{interval modules}. We say that a persistence module $M$ has a \emph{barcode} if there exists an index set $A$ and a collection of intervals $(I_{\alpha})_{\alpha\in A}$ such that 
\[
M\cong\bigoplus_{\alpha\in A}C(I_{\alpha}).
\]
By the Krull-Remak-Schmidt-Azumaya Theorem \cite{MR37832}, this collection of intervals is unique up to reordering if it exists. The most important existence result for barcodes is Crawley-Boevey's Theorem \cite{MR3323327,MR4143378}, which states that every p.f.d.\@ persistence module has a barcode. It is well-known that this does not extend to q-tame persistence modules. However, as shown by Chazal et al.\@ in \cite{MR3575998}, q-tame persistence modules can be decomposed as direct sums of interval modules up to \emph{weak isomorphism}, where a morphism $\varphi\colon M\to N$ of persistence modules is called a weak isomorphism if $\ker\varphi$ and $\coker\varphi$ are ephemeral.

Using the techniques developed by Chazal et al., we will show that, under some mild assumptions on the index set, an interesting class of q-tame persistence modules actually admit decompositions into interval modules up to isomorphism and not just weak isomorphism.

\begin{defi}
A persistence module $M$ is called \emph{upper semi-continuous (u.s.c.)} if the canonical map
\[
M_t\to\lim_{s>t} M_{s}
\]
 is an isomorphism for all $t\in T$. It is called \emph{lower semi-continuous (l.s.c.)} if the canonical map
 \[
 \colim_{s<t} M_{s}\to M_t
 \]
 is an isomorphism for all $t\in T$.
\end{defi}

While not explicitly stated by Chazal et al., the next result is an immediate corollary of \cite[Corollary 3.6.]{MR3575998}. The terms involving the index set will be introduced in \cref{defi:index_set}. In the important special case $T=\mathbb{R}$, all assumptions are satisfied.

\begin{thm}\label{thm:lsc}
Let $T$ be a dense totally ordered set such that every interval in $T$ has a countable coinitial subset. Then every q-tame lower semi-continuous persistence module indexed by $T$ has a barcode.
\end{thm}

With some additional work, we will prove the following novel result.

\begin{thm}\label{thm:usc}
Let $T$ be a dense totally ordered set such that every interval in $T$ has a countable coinitial subset. Then for every q-tame upper semi-continuous persistence module $M$ indexed by $T$ there exists a collection of intervals $(I_{\alpha})_{\alpha\in A}$, unique up to reordering, such that 
\[
M\cong\prod_{\alpha\in A}C(I_{\alpha}).
\]
\end{thm}

In general, uniqueness statements for product decompositions are much harder to come by than in the case of direct sums where one has the Krull-Remak-Schmidt-Azumaya Theorem. We will also infer our uniqueness statement from this theorem, rather than from a general statement about products.

In order to distinguish product and direct sum decompositions, we suggest the following terminology.

\begin{defi}
We say that a persistence module has an \emph{additive barcode} if it is isomorphic to the direct sum of a collection of interval modules that is unique up to reordering, i.e., if it has a barcode in the usual sense. We say that it has a \emph{multiplicative barcode} if it is isomorphic to a product of a collection of interval modules that is unique up to reordering. 
\end{defi}

In the p.f.d.\@ case, the two notions agree. In the q-tame case, however, there are persistence modules that have a multiplicative barcode, but no additive barcode and vice versa (\cref{ex:semi_cont}).

\section{Preliminaries}\label{sec:background}
Let us begin by introducing some terminology.
\begin{defi}\label{defi:index_set}
A totally ordered set $T$ is called \emph{dense} if for all $s,t\in T$ with $s<t$ there exists $u\in T$ with $s<u<t$. If $N\subseteq I\subseteq T$ are subsets, $N$ is said to be \emph{coinitial} in $I$ if for all $t\in I$ there exists $s\in N$ with $s\leq t$. $N$ is said to be \emph{cofinal} in $I$ if for all $t\in I$ there exists $s\in N$ with $t\leq s$.
\end{defi}

A central tool in proving our results is the radical of a persistence module.

\begin{defi}{{\cite[Definition 2.10., Remark 2.12., Definition 3.4.]{MR3575998}}}
If $M$ is a persistence module, we define a persistence module $\underline{M}$ by 
\[
\underline{M}_{t}=\lim_{s>t}M_{s}
\]
with the obvious structure maps. The canonical maps $M_{t}\to\lim_{s>t}M_{s}$ form a morphism $M\to\underline{M}$. We also define a persistence module $\overline{M}$ by 
\[
\overline{M}_{t}=\colim_{s<t}M_{s},
\]
again with the obvious structure maps. The canonical maps $\colim_{s<t}M_{s}\to M_{t}$ form a morphism $\overline{M}\to M$. We define the \emph{radical} of $M$ as 
\[
\rad M=\im(\overline{M}\to M).
\]
The assignments $\underline{(-)}$, $\overline{(-)}$ and $\rad(-)$ extend to endofunctors on the category of persistence modules by acting on morphisms via the universal properties of limits, colimits and images.
\end{defi}

By convention, we assume limits and colimits over empty index sets to be $0$ in the definition above. The radical of a q-tame persistence module need not be p.f.d.\@ (\cite[Example 3.7.]{MR3575998}). Still, a barcode existence theorem involving descending chain conditions for images and kernels of structure maps yields the following.

\begin{thm}\label{thm:rad}{\cite[Corollary 3.6.]{MR3575998}}
Let $T$ be a dense totally ordered set such that every interval in $T$ has a countable coinitial subset. If $M$ is a q-tame persistence module indexed by $T$, its radical $\rad M$ has a barcode.
\end{thm}

This is already all we need in order to prove that l.s.c.\@ q-tame persistence modules have a barcode.

\begin{proof}[Proof of \cref{thm:lsc}]
By definition, a persistence module $M$ is lower semi-continuous if the canonical morphism $\overline{M}\to M$ is an isomorphism. In particular, a lower semi-continuous persistence module is isomorphic to its radical. Thus, the claim is an immediate consequence of \cref{thm:rad}.
\end{proof}

Next, we will analyse how the functors defined above behave on interval modules.

\begin{defi}
For $t\in T$, we write
\begin{align*}
\uparrow t&=\{s\in T\mid s> t\}\\
\downarrow t&=\{s\in T\mid s<t\}
\end{align*}
for the strict upset and the strict downset of $t$. If $I\subseteq T$ is an interval, we define
\begin{align*}
\underline{I}&=\left\{t\in T\mid I\cap\uparrow t\text{ is non-empty and coinitial in }\uparrow t\right\} \\
\overline{I}&=\left\{t\in T\mid I\cap\downarrow t\text{ is non-empty and cofinal in }\downarrow t\right\} \\
\rad I&=I\cap\overline{I}
\end{align*}
\end{defi}

\begin{lem}
Let $I\subseteq T$ be an interval. Then the sets $\underline{I}$, $\overline{I}$ and $\rad I$ are again intervals in $T$ if they are non-empty.
\end{lem}
\begin{proof}
We only show the claim for $\underline{I}$, the other ones can be shown similarly. Let $s,u\in \underline{I}$ and $t\in T$ with $s<t<u$. We need to show that $t\in\underline{I}$, i.e. for $a\in\uparrow t$ we need to find $b\in I\cap\uparrow t$ with $b\leq a$.

First, we show that $u\in I$: Since $u\in\underline{I}$, there exists some $v\in I\cap\uparrow u$. Since $u\in\uparrow s$ and $s\in\underline{I}$, there exists $c\in I\cap\uparrow s$ with $c\leq u$. We have $c\leq u<v$ and $c,v\in I$. Since $I$ is an interval, we get $u\in I$. In particular, we have $u\in I\cap\uparrow t$, so this set is non-empty.

Now consider $a\in\uparrow t$ again. We have $u\in\ I\cap\uparrow t$, so if $u\leq a$, we can set $b=u$ and are done. If $a< u$, pick $c\in I\cap\uparrow s$ with $c\leq a$. This is possible since $a\in\uparrow t\subseteq \uparrow s$ and $s\in\underline{I}$. Then, we have $c\leq a< u$ and $c,u\in I$, which implies $a\in I$. So in this case, we can simply set $b=a$ and the proof is finished.
\end{proof}

For the proof of our main theorem, we will need the following.

\begin{lem}\label{lem:usc_rad}
Let $T$ be a dense totally ordered set and $I\subseteq T$ an interval. If $\underline{I}=I$, then $\rad I$ is non-empty and $\underline{\rad I}=I$.
\end{lem}
\begin{proof}
First, note that $\underline{I}=I$ implies that $\underline{I}$ is non-empty. In other words, there exists $t\in T$ such that $I\cap\uparrow t\neq\emptyset$ is coinitial in $\uparrow t$. We will show that $I\cap\uparrow t=\rad I\cap\uparrow t$. This immediately implies that $\rad I$ is non-empty. It also shows that $\underline{I}\subseteq\underline{\rad I}$. The other inclusion obviously also holds, so in total we get $\underline{\rad I}=\underline{I}=I$ as claimed.

It is clear that $I\cap\uparrow t\supseteq\rad I\cap\uparrow t$. To see the other inclusion, consider $s\in I\cap\uparrow t$. We need to show that $s\in\rad I=I\cap\overline{I}$, so it is enough to check that $s\in\overline{I}$. So let $a\in\downarrow s$. We need to find $b\in I\cap\downarrow s$ with $b\geq a$. 

If $a>t$, we have $a\in I$: Since $I\cap\uparrow t$ is coinitial in $\uparrow t$, we may choose $s'\in I\cap\uparrow t$ with $s'\leq a$. Now $s'\leq a<s$ and $s,s'\in I$, so $a\in I$ because $I$ is an interval. In this case, we can set $b=a$ and are done.

If $a\leq t$, we use the fact that $T$ is dense to choose $c\in T$ with $t<c<s$. By the same argument as before, we get $c\in I$ and can set $b=c$. This finishes the proof.
\end{proof}

Recall that for an interval $I$, we denote the corresponding interval module as defined in the introduction by $C(I)$. While we do not consider the empty set to be an interval, we set $C(\emptyset)=0$. Then, the lemma below still holds true if the involved sets are empty.

\begin{lem}\label{lem:over_under_rad_int}
For any interval $I\subseteq T$ we have
\begin{align*}
\underline{C(I)}&\cong C(\underline{I})\\ 
\overline{C(I)}&\cong C(\overline{I})\\
\rad C(I)&\cong C(\rad I)
\end{align*} 
\end{lem}
\begin{proof}
Again, we only show the first isomorphism and the others can be shown analogously. For all $t\in\underline{I}$, we have
\[
\underline{C(I)}_{t}=\lim_{s\in\uparrow t} C(I)_{s}=\lim_{s\in I\cap\uparrow t} C(I)_{s}=\lim_{s\in I\cap\uparrow t} \mathbb{F}=\mathbb{F}.
\]
For $t\notin\underline{I}$, we have that $I\cap\uparrow t$ is empty or that there exists $t_{0}\in\uparrow{t}$ such that there is no $s\in I\cap\uparrow t$ with $s\leq t_{0}$. In the first case, we have 
\[
\underline{C(I)}_{t}=\lim_{s\in\uparrow t} C(I)_{s}=\lim_{s\in\uparrow t} 0=0
\]
In the second case, we have
\[
\underline{C(I)}_{t}=\lim_{s\in\uparrow t} C(I)_{s}=\lim_{t<s\leq t_{0}} C(I)_{s}=\lim_{t<s\leq t_{0}}0=0.
\]
Thus, $\underline{C(I)}$ and $C(\underline{I})$ agree pointwise. Clearly, their structure maps also agree and we obtain the claim.
\end{proof}

\section{Semi-Continuous Persistence Modules}\label{sec:main}
Recall that a persistence module is u.s.c.\@ if the canonical morphism $M\to\underline{M}$ is an isomorphism and l.s.c.\@ if the canonical morphism $\overline{M}\to M$ is an isomorphism. We start with a basic observation for direct sums and products.

\begin{lem}\label{lem:prod_sum_sc}
Let $(M_{\alpha})_{\alpha\in A}$ be a collection of persistence modules. 
\begin{enumerate}
	\item $\bigoplus_{\alpha\in A} M_{\alpha}$ is l.s.c.\@ if and only if all $M_{\alpha}$ are l.s.c.
	\item $\prod_{\alpha\in A} M_{\alpha}$ is u.s.c.\@ if and only if all $M_{\alpha}$ are u.s.c.
\end{enumerate}
\end{lem}
\begin{proof}
It is easy to check that taking direct sums of persistence modules is conservative, so the canonical map $\bigoplus_{\alpha\in A} \overline{M_{\alpha}}\to\bigoplus_{\alpha\in A} M_{\alpha}$ is an isomorphism if and only if all $M_{\alpha}$ are l.s.c. Colimits commute with each other, so we also have a canonical isomorphism 
\[
\overline{\bigoplus_{\alpha\in A} M_{\alpha}}\cong\bigoplus_{\alpha\in A} \overline{M_{\alpha}}.
\]
This implies the first claim. The second claim follows analogously because taking products of persistence modules is also conservative and limits commute with each other.
\end{proof}

Semi-continuity is also easy to characterize for interval modules.

\begin{lem}\label{lem:int_sc}
Let $I\subseteq T$ be an interval.
\begin{enumerate}
	\item $C(I)$ is l.s.c.\@ if and only if $I=\overline{I}$.
	\item $C(I)$ is u.s.c.\@ if and only if $I=\underline{I}$.
\end{enumerate}
\end{lem}
\begin{proof}
Both claims follow immediately from \cref{lem:over_under_rad_int} and the fact that for any two interval modules $C(J)$ and $C(J')$ we have $C(J)\cong C(J')$ if and only if $J=J'$.
\end{proof}

In particular, the two lemmas imply that if an l.s.c.\@ persistence module has an additive barcode, then all intervals $I$ appearing in this barcode satisfy $I=\overline{I}$. Similarly, if a u.s.c.\@ persistence module has a multiplicative barcode, then all intervals $I$ in this barcode satisfy $I=\underline{I}$. 

As an illustration, observe that a real interval $I\subseteq\mathbb{R}$ satisfies $I=\overline{I}$ if and only if $I=\mathbb{R}$ or $I=[a,b)$ for some $a\in\mathbb{R}$ and $b\in\mathbb{R}\cup\{\infty\}$. Analogously, we have $I=\underline{I}$ if and only if $I=\mathbb{R}$ or $I=(a,b]$ for some $a\in\mathbb{R}\cup\{-\infty\}$ and $b\in\mathbb{R}$.

Semi-continuous persistence modules appear naturally in many different contexts. Some authors, especially within symplectic topology, even go as far as to consider almost exclusively semi-continuous persistence modules, see e.g.\@ \cite{MR3437837}.

\begin{ex}\label{ex:semi_cont}
\begin{itemize}
	\item One of the standard examples of a q-tame persistence module indexed by $\mathbb{R}$ that does not have a barcode in the usual sense is 
	\[
	\prod_{n\in\mathbb{N}} C\left(\left[0,n^{-1}\right)\right).
	\]
	It is upper semi-continuous by the previous two lemmas, so it has a multiplicative barcode barcode by \cref{thm:usc}. Clearly, this is given by $([0,n^{-1}))_{n\in\mathbb{N}}$. In particular, this persistence module has a multiplicative barcode but no additive barcode.
	\item Consider the $\mathbb{R}$-indexed persistence module 
	\[
	\bigoplus_{n\in\mathbb{N}}C\left(\left(-n^{-1},0\right]\right).
	\]
	It is lower semi-continuous by the previous two lemmas and also q-tame. It has an additive barcode but no multiplicative barcode.
	\item Let $X\colon\mathbf{T}\to\mathbf{Top}$ be a diagram of topological spaces. If $X_{t}$ is a compact Hausdorff space for all $t\in T$ and $X$ is upper semi-continuous, i.e., $X_{t}\to\lim_{s>t} X_{s}$ is an isomorphism for all $t\in T$, then the persistent \v{C}ech homology $\CH_{*}(X,\mathbb{F})$ is upper semi-continuous by \cite[Theorem X.3.1.]{MR0050886}. This persistent \v{C}ech homology is studied by Morse in a project extending his calculus of variations in the large to the study of minimal surfaces, see e.g. \cite{MR1503341,MR1503471,MR1927}. To our knowledge, this is the first instance of persistent homology in the mathematical literature.
	\item Let $X$ be a topological space and $f\colon X\to\mathbb{R}$ a continuous map. Write $f_{<t}$ for the open sublevel set of $f$ at $t$. Since $f$ is continuous, we have 
	$
	f_{<t}=\colim_{s<t}f_{<s}
	$
	for all $t$, where the colimit is taken in the category of topological spaces. Using the fact that the interval $(-\infty,t)$ has a countable cofinal subset, the main theorem in \cite[Section 14.6.]{MR1702278} implies that the sublevel set persistence $H(f_{<\bullet})$ is lower semi-continuous. Here, $H$ is any generalized homology theory with values in $\mathbf{Vec}_{\mathbb{F}}$.
	\item For any persistence module $M$ indexed by $\mathbf{T}$, we get a dual persistence module $M^*$ indexed by $\mathbf{T}^{\mathrm{op}}$ defined by composing the functor $M\colon\mathbf{T}\to\mathbf{Vec}_{\mathbb{F}}$ with the contravariant functor $\Hom(-,\mathbb{F})\colon\mathbf{Vec}_{\mathbb{F}}\to\mathbf{Vec}_{\mathbb{F}}$. If $M$ is lower semi-continuous, then $M^*$ is upper semi-continuous:
	\begin{align*}
	\underline{M^*}_{t}&=\lim_{s<t}\Hom(M_{s},\mathbb{F})\\&=\Hom(\colim_{s<t}M_{s},\mathbb{F})\\&=\Hom(M_{t},\mathbb{F})\\&=M^*_{t},
	\end{align*}
	where equality should be interpreted as 'canonically isomorphic'. 
	
	However, if $M$ is upper semi-continuous, $M^*$ need \emph{not} be lower semi-continuous: Consider $M=\prod_{n\in\mathbb{N}} C\left(\left[0,n^{-1}\right)\right)$ as in the first example. An easy calculation shows that  
	\[
	\overline{M^*}_{0}\cong\bigoplus_{n\in\mathbb{N}}\mathbb{F},
	\]
	but we also have
	\[
	M^*_{0}=\Hom\left(\prod_{n\in\mathbb{N}}\mathbb{F},\mathbb{F}\right),
	\]
	which is isomorphic to the double dual space of $\bigoplus_{n\in\mathbb{N}}\mathbb{F}$. Since no infinite dimensional vector space is isomorphic to its double dual, we obtain that $M^*$ is not lower semi-continuous at $0$.
\end{itemize}
\end{ex}

Finally, we want to prove our decomposition theorem for the u.s.c.\@ case. An essential fact for our proof is that in the q-tame case, direct sums and products of persistence modules do not differ too much. Recall that a morphism of persistence modules is called a weak isomorphism if its kernel and cokernel are ephemeral.

\begin{prop}\label{lem:sum_prod}
Let $(M_{\alpha})_{\alpha\in A}$ be a collection of persistence modules such that $\prod_{\alpha\in A}M_{\alpha}$ is q-tame. Then the canonical map
\[
\bigoplus_{\alpha\in A}M_{\alpha}\to\prod_{\alpha\in A}M_{\alpha}
\]
is a weak isomorphism.
\end{prop}
\begin{proof}
Denote the map above by $\varphi$. Clearly, $\varphi$ has trivial kernel. Thus, it suffices to show that $\coker\varphi$ is ephemeral. So let $s,t\in T$ with $s<t$ and consider the diagram
\[
\begin{tikzcd}
\left(\bigoplus\limits_{\alpha\in A}M_{\alpha}\right)_{t}\arrow[r,"\varphi_{t}",hook]&\left(\prod\limits_{\alpha\in A}M_{\alpha}\right)_{t}\arrow[r,"p_{t}",two heads]&\coker\varphi_{t}\\
\left(\bigoplus\limits_{\alpha\in A}M_{\alpha}\right)_{s}\arrow[r,"\varphi_{s}",hook]\arrow[u,"\sigma_{s,t}"]&\left(\prod\limits_{\alpha\in A}M_{\alpha}\right)_{s}\arrow[r,"p_{s}",two heads]\arrow[u,"\pi_{s,t}"]&\coker\varphi_{s}\arrow[u,"\gamma_{s,t}"]
\end{tikzcd}
\]
where we added some shorthand notation for the structure maps of the persistence modules we consider. We need to check that $\gamma_{s,t}=0$. Since $p_{s}$ is epi, it is enough to show $\gamma_{s,t}\circ p_{s}=0$. Commutativity of the above diagram implies that 
\[
\gamma_{s,t}\circ p_{s}=p_{t}\circ\pi_{s,t}.
\]
Note that $p_{t}\circ\varphi_{t}=0$, so we are done if we can show that $\pi_{s,t}$ factors through $\varphi_{t}$. To see that this is the case, we factor $\sigma_{s,t}$ and $\pi_{s,t}$ through their images to obtain a diagram
\[
\begin{tikzcd}
\left(\bigoplus\limits_{\alpha\in A}M_{\alpha}\right)_{t}\arrow[r,"\varphi_{t}",hook]&\left(\prod\limits_{\alpha\in A}M_{\alpha}\right)_{t}\\
\im\sigma_{s,t}\arrow[r,"\psi_{s,t}"]\arrow[u,hook]&\im\pi_{s,t}\arrow[u,hook]\\
\left(\bigoplus\limits_{\alpha\in A}M_{\alpha}\right)_{s}\arrow[r,"\varphi_{s}",hook]\arrow[u,two heads]&\left(\prod\limits_{\alpha\in A}M_{\alpha}\right)_{s}\arrow[u,two heads]
\end{tikzcd}
\]
We can canonically identify
\[
\im\sigma_{s,t}\cong\bigoplus_{\alpha\in A}\im (M_{\alpha})_{s,t}
\]
and
\[
\im\pi_{s,t}\cong\prod_{\alpha\in A}\im (M_{\alpha})_{s,t}.
\]
From commutativity of the previous diagram, it is easy to see that under this identification $\psi_{s,t}$ is simply the canonical inclusion of the direct sum into the product. But here, this map is an isomorphism since $\im\pi_{s,t}$ is finite dimensional by our q-tameness assumption. Thus, we can invert $\psi_{s,t}$, yielding a factorization of $\pi_{s,t}$ as 
\[
\begin{tikzcd}
\left(\prod\limits_{\alpha\in A}M_{\alpha}\right)_{s}\arrow[r,two heads]&\im\pi_{s,t}\arrow[r,"\psi_{s,t}^{-1}",hook, two heads]&\im\sigma_{s,t}\arrow[r,hook]&\left(\bigoplus\limits_{\alpha\in A}M_{\alpha}\right)_{t}\arrow[r,"\varphi_{t}",hook]&\left(\prod\limits_{\alpha\in A}M_{\alpha}\right)_{t}
\end{tikzcd}
\]
As explained above, this finishes the proof.
\end{proof}

Before proceeding to the main proof, we show two more lemmas.

\begin{lem}\label{lem:weak_iso_overline}
Let $\varphi\colon M\to N$ be a weak isomorphism of persistence modules. Then $\varphi$ induces an isomorphism
\[
\overline{\varphi}\colon\overline{M}\to\overline{N}.
\]
\end{lem} 
\begin{proof}
Since taking direct limits of vector spaces is exact, the same is true for the functor $\overline{(-)}$. Thus, this functor commutes with kernels and cokernels, so we get
\[
\ker\overline{\varphi}\cong\overline{\ker\varphi}
\]
and
\[
\coker\overline{\varphi}\cong\overline{\coker\varphi}
\]
Since $\ker\varphi$ and $\coker\varphi$ are ephemeral by assumption, we get that in both cases the right-hand side vanishes. So $\overline{\varphi}$ has trivial kernel and cokernel, which proves the claim.
\end{proof}

\begin{lem}\label{lem:weak_iso_underline}
Assume that every interval in $T$ has a countable coinitial subset. Let $\varphi\colon M\to N$ be a weak isomorphism of persistence modules. Then $\varphi$ induces an isomorphism
\[
\underline{\varphi}\colon\underline{M}\to\underline{N}.
\]
\end{lem}
\begin{proof}
Consider the epi-mono-factorization of $\varphi$ as
\[
\begin{tikzcd}
M\arrow[r,two heads,"p"]&\im\varphi\arrow[r,hook,"i"]&N
\end{tikzcd}
\]
In order to show that $\underline{\varphi}$ is an isomorphism, it suffices to prove that $\underline{p}$ and $\underline{i}$ are isomorphisms. First, consider the short exact sequence
\[
\begin{tikzcd}
0\arrow[r]&\im\varphi\arrow[r,"i"]&N\arrow[r]&\coker\varphi\arrow[r]&0
\end{tikzcd}
\]
Since taking inverse limits of vector spaces is left-exact, the functor $\underline{(-)}$ is also left-exact. Thus, we get an exact sequence
\[
\begin{tikzcd}
0\arrow[r]&\underline{\im\varphi}\arrow[r,"\underline{i}"]&\underline{N}\arrow[r]&\underline{\coker\varphi}
\end{tikzcd}
\]
By assumption, $\coker\varphi$ is ephemeral, so we have $\underline{\coker\varphi}=0$, which implies that $\underline{i}$ is an isomorphism. Next, consider the short exact sequence
\[
\begin{tikzcd}
0\arrow[r]&\ker\varphi\arrow[r]&M\arrow[r,"p"]&\im\varphi\arrow[r]&0
\end{tikzcd}
\]
For each $t\in T$, the interval $\{s\in T\mid s>t\}$ has a countable coinitial subset by assumption. Since $\ker\varphi$ is ephemeral, the inverse system $(\ker\varphi_{s})_{s>t}$ satisfies the Mittag-Leffler property for all $t\in T$. Thus, by \cite[Proposition 13.2.2.]{MR217085} the sequence
\[
\begin{tikzcd}
0\arrow[r]&\lim\limits_{s>t}\ker\varphi_{s}\arrow[r]&\lim\limits_{s>t}M_{s}\arrow[r]&\lim\limits_{s>t}\im\varphi_{s}\arrow[r]&0
\end{tikzcd}
\]
is exact for all $t\in T$. Consequently, the sequence
\[
\begin{tikzcd}
0\arrow[r]&\underline{\ker\varphi}\arrow[r]&\underline{M}\arrow[r,"\underline{p}"]&\underline{\im\varphi}\arrow[r]&0
\end{tikzcd}
\]
is also exact. We have $\underline{\ker\varphi}=0$ since $\ker\varphi$ is assumed to be ephemeral. Hence, $\underline{p}$ is an isomorphism and the proof is finished.
\end{proof}

\begin{rem}
The previous lemma also holds if we replace the assumption on $T$ by the assumption that $T$ be a dense order. In this case, the lemma is a consequence of the fact that $\underline{(-)}$ defines a functor on the observable category of persistence modules and that weak isomorphisms turn to isomorphisms when mapped to the observable category (\cite[Remark 2.12., Theorem 2.9.]{MR3575998}).
\end{rem}

\begin{proof}[Proof of \cref{thm:usc}]
Under our assumptions, $\rad M$ has a barcode by \cref{thm:rad}, say $(I_{\alpha})_{\alpha\in A}$.  We claim that $M$ is isomorphic to the product over the interval modules $C(\underline{I_{\alpha}})$. 

First, we have
\[
M\cong\underline{M}
\]
since we assume $M$ to be u.s.c. Since the canonical map $\rad M\to M$ is a weak isomorphism (as a consequence of \cite[Proposition 2.11.]{MR3575998}), \cref{lem:weak_iso_underline} implies
\[
\underline{M}\cong\underline{\rad M}.
\] 
Recall that $\underline{(-)}$ is a functor, so 
\[
\underline{\rad M}\cong\underline{\bigoplus_{\alpha\in A}C(I_{\alpha})}
\]
because the barcode of $\rad M$ is given by the $I_{\alpha}$. The inclusion of the direct sum into the product is a weak isomorphism in the q-tame case (\cref{lem:sum_prod}), so \cref{lem:weak_iso_underline} implies
\[
\underline{\bigoplus_{\alpha\in A}C(I_{\alpha})}\cong\underline{\prod_{\alpha\in A}C(I_{\alpha})}.
\]
Since limits commute with products we also get
\[
\underline{\prod_{\alpha\in A}C(I_{\alpha})}\cong\prod_{\alpha\in A}\underline{C(I_{\alpha})}.
\]
We have $\underline{C(I_{\alpha})}\cong C(\underline{I_{\alpha}})$ by \cref{lem:over_under_rad_int}, so that 
\[
\prod_{\alpha\in A}\underline{C(I_{\alpha})}\cong\prod_{\alpha\in A}C(\underline{I_{\alpha}}).
\]
Putting everything together yields that $M$ is indeed isomorphic to the product over the $C(\underline{I_{\alpha}})$.

The uniqueness part of the statement essentially follows by reversing the above argument. Suppose $(J_{\beta})_{\beta\in B}$ are also intervals such that
\[
M\cong\prod_{\beta\in B}C(J_{\beta}).
\] 
We want to prove that $(J_{\beta})_{\beta\in B}$ and $(\underline{I_{\alpha}})_{\alpha\in A}$ agree up to reordering. Note that this in particular implies that each $\underline{I_{\alpha}}$ is non-empty. Since $M\cong\prod_{\beta} C(J_{\beta})$ is u.s.c.\@ each factor $C(J_{\beta})$ must be u.s.c.\@ as well by \cref{lem:prod_sum_sc}. Together with \cref{lem:int_sc} this yields
\[
\underline{J_{\beta}}=J_{\beta}.
\]
Thus, by \cref{lem:usc_rad} we get that $\rad J_{\beta}$ is non-empty and consequently an interval for all $\beta$. Next, we will show that $(\rad J_{\beta})_{\beta\in B}$ is a barcode for $\rad M$: Consider
\[
\rad M\cong\rad\prod_{\beta\in B}C(J_{\beta})=\im\left(\overline{\prod_{\beta\in B}C(J_{\beta})}\to\prod_{\beta\in B}C(J_{\beta})\right).
\] 
Recall that the inclusion of the direct sum into the product is a weak isomorphism in our case, so together with \cref{lem:weak_iso_overline} we obtain that 
\[
\im\left(\overline{\prod_{\beta\in B}C(J_{\beta})}\to\prod_{\beta\in B}C(J_{\beta})\right)\cong\im\left(\overline{\bigoplus_{\beta\in B}C(J_{\beta})}\to\prod_{\beta\in B}C(J_{\beta})\right),
\]
where the map on the right is equal to the composition of the natural map \[\overline{\bigoplus_{\beta\in B}C(J_{\beta})}\to\bigoplus_{\beta\in B}C(J_{\beta})\] and the inclusion $\bigoplus_{\beta\in B}C(J_{\beta})\to\prod_{\beta\in B}C(J_{\beta})$. Since this inclusion is mono, we get
\[
\im\left(\overline{\bigoplus_{\beta\in B}C(J_{\beta})}\to\prod_{\beta\in B}C(J_{\beta})\right)\cong\im\left(\overline{\bigoplus_{\beta\in B}C(J_{\beta})}\to\bigoplus_{\beta\in B}C(J_{\beta})\right).
\]
Direct sums and the functor $\overline{(-)}$ commute. The same is true for direct sums and images, so we get
\[
\im\left(\overline{\bigoplus_{\beta\in B}C(J_{\beta})}\to\bigoplus_{\beta\in B}C(J_{\beta})\right)\cong\bigoplus_{\beta\in B}\im(\overline{C(J_{\beta})}\to C(J_{\beta}))=\bigoplus_{\beta\in B}\rad C(J_{\beta}).
\]
We have $\rad C(J_{\beta})\cong C(\rad J_{\beta})$ by \cref{lem:over_under_rad_int}, so we get
\[
\bigoplus_{\beta\in B}\rad C(J_{\beta})\cong \bigoplus_{\beta\in B}C(\rad J_{\beta}).
\] 
In total, we have shown that $(\rad J_{\beta})_{\beta\in B}$ is indeed a barcode for $\rad M$. 

Using the Krull-Remak-Schmidt-Azumaya Theorem, we obtain that $(I_{\alpha})_{\alpha\in A}$ and $(\rad J_{\beta})_{\beta\in B}$ agree up to reordering. This implies that also $(\underline{I_{\alpha}})_{\alpha\in A}$ and $(\underline{\rad J_{\beta}})_{\beta\in B}$ agree up to reordering. Now recall that we have $\underline{J_{\beta}}=J_{\beta}$ for all $\beta$ because $M$ is u.s.c. By \cref{lem:usc_rad}, we get that $(\underline{\rad J_{\beta}})_{\beta\in B}=(J_{\beta})_{\beta}$. Thus, $(\underline{I_{\alpha}})_{\alpha\in A}$ and $(J_{\beta})_{\beta\in B}$ agree up to reordering. This finishes the proof.
\end{proof}

\section*{Acknowledgments}
The author thanks Ulrich Bauer for his invaluable guidance during the preparation of the thesis that the results in this paper grew out of. Peter Albers and Lucas Dahinden also provided helpful comments on this manuscript. 
This research was supported by Deutsche Forschungsgemeinschaft (DFG, German Research Foundation) through Germany’s Excellence Strategy EXC-2181/1 - 390900948 (the Heidelberg STRUCTURES Excellence Cluster), the Transregional Colloborative Research Center CRC/TRR 191 (281071066) and the Research Training Group RTG 2229 (281869850).

\providecommand{\bysame}{\leavevmode\hbox to3em{\hrulefill}\thinspace}
\providecommand{\MR}{\relax\ifhmode\unskip\space\fi MR }
\providecommand{\MRhref}[2]{%
  \href{http://www.ams.org/mathscinet-getitem?mr=#1}{#2}
}
\providecommand{\href}[2]{#2}

\end{document}